\documentclass[11pt]{amsart}
\usepackage{fullpage}

\usepackage{amsmath}
\usepackage{amsfonts}
\usepackage{amssymb}
\usepackage{amsthm}
\usepackage{subfigure}
\usepackage{graphicx}
\usepackage{color}
\usepackage{enumerate}
\usepackage{lineno}
\usepackage{pinlabel}
\newtheorem{theorem}{Theorem}[section]
\newtheorem{proposition}[theorem]{Proposition}
\newtheorem{lemma}[theorem]{Lemma}

\newtheorem{corollary}[theorem]{Corollary}

\theoremstyle{definition}
\newtheorem{definition}[theorem]{Definition}

\theoremstyle{remark}

\numberwithin{equation}{section}

\DeclareMathOperator{\cut}{\mathrm{Cut}}

\begin{document}

%\renewcommand{\labelenumi}{(\arabic{enumi})}
%\renewcommand{\labelenumii}{(\roman{enumii})}
%\renewcommand{\labelenumi}{(\roman{enumi})}

%names to be decided upon:
%\newcommand{\MBR}{multivariate Bollob\'{a}s-Riordan }
%\newcommand{\BR}{Bollob\'{a}s-Riordan }
%\newcommand{\MBR}{topological dichromatic }
%\newcommand{\la}{\lambda}

%\linenumbers

\title{The Las Vergnas Polynomial for embedded graphs}

\dedicatory{This paper is dedicated to the memory of Michel Las Vergnas in gratitude for not only so much beautiful mathematics, but also many instances of very kind and  insightful correspondence.}

\author[J.~Ellis-Monaghan]{Joanna A. Ellis-Monaghan}
\address{Department of Mathematics, Saint Michael's College, 1 Winooski Park, Colchester, VT 05439, USA.  }
\email{jellis-monaghan@smcvt.edu}

\author[I.~Moffatt]{Iain Moffatt}
\address{	Department of Mathematics,
Royal Holloway,
University of London,
Egham,
Surrey,
TW20 0EX,
United Kingdom.}
\email{iain.moffatt@rhul.ac.uk}

\date{\today}

\begin{abstract}
The Las Vergnas polynomial is an extension of the Tutte polynomial to cellularly embedded graphs. It was introduced by  Michel Las~Vergnas in  1978 as special case of his Tutte polynomial of a morphism of matroids.  While the general Tutte polynomial of a morphism of matroids has a complete set of deletion-contraction relations, its specialisation to cellularly embedded graphs does not.
Here we  extend the Las Vergnas polynomial to graphs in pseudo-surfaces.  We show that in this setting we can define deletion and contraction for embedded graphs consistently with the deletion and contraction of the underlying matroid perspective, thus yielding a version of the Las Vergnas polynomial with complete recursive definition.
 This also enables us to obtain a deeper understanding of the relationships among the Las Vergnas polynomial, the Bollob\'as-Riordan polynomial, and the Krushkal polynomial.   We also take this opportunity to extend some of  Las Vergnas'  results on Eulerian circuits from graphs in surfaces of low genus to surfaces of arbitrary genus.
\end{abstract}

\maketitle

\section{introduction}

In \cite{Las78a,Las80}  (see also \cite{Las75,Las78}), Michel Las Vergnas introduced a polynomial $L_G(x,y,z)$ that extends the classical Tutte polynomial to cellularly embedded graphs. This topological Tutte polynomial, now called the Las Vergnas polynomial,  is the first extension of the Tutte polynomial to embedded graphs that the authors are aware of. Michel Las Vergnas was ahead of his time in his investigation as not until many years later did other mathematicians and physicists initiate the serious attention now paid to embedded graph polynomials.  More recent embedded graph polynomials, such as the ribbon graph polynomial of Bollob\'as and Riordan, $R_G$ (see \cite{BR1,BR2}),  and Krushkal's polynomial, $K_G$ (see \cite{Kr}), have led in turn to renewed interest in $L_G$, for example in \cite{ACEMS, But}.

  The Las Vergnas polynomial was first defined in terms of the combinatorial geometry of an embedded graph (i.e., via circuit matroids). It arises as  a special case in his much larger body of work on the Tutte polynomial of a morphism of matroids (see \cite{EL04,Las80,Las99,Las84,Las07,Las12}).  We present here a discussion of matroid perspectives in the special context of embedded graph theory.  Although $L_G$ has its origins in matroid theory, it is of independent interest as a tool for extracting both combinatorial and topological information from graphs embedded in surfaces.   Accordingly, one of the aims of this work is to provide a formulation of $L_G$ that is readily accessible to topological graph theorists without reference to matroid theory, and so to encourage further investigation into it.  (Also see \cite{ACEMS} for such a discussion.)

We are especially interested here in deletion-contraction definitions of graph polynomials. A very desirable property of such a recursive definition is that it reduces any graph to a linear combination of edgeless graphs.   Las Vergnas gave this type of  deletion-contraction definition for the Tutte polynomial of a morphism of matroids.  This definition, however, does not hold for his cellularly embedded graph polynomial $L_G$.

We show that by using an appropriate matroid framework to extend the  Las Vergnas polynomial to  graphs in pseudo-surfaces  (but not necessarily cellularly embedded graphs) it is possible to obtain a deletion-contraction definition of $L_G$ in the language of topological graph theory. Furthermore, this recursive definition for the embedded graph polynomial is consistent with that for the Tutte polynomial of a morphism of matroids. Our approach begins by associating an abstract graph $G^{\dagger}$  to a graph in a pseudo-surface in analogy with the construction of $G^*$ for a cellularly embedded graph $G$.  We then see that the bond matroid of $G^{\dagger}$  measures how a graph separates the pseudo-surface it is embedded in  the same way as the bond matroid of $G^*$ does. This matroid allows us to extend the Las Vergnas polynomial to  the broader class of graphs in pseudo-surfaces. Moreover, this extended polynomial arises as a special case Tutte polynomial of a morphism of matroids, just as the original polynomial for cellularly embedded graphs did.
  By  using a deletion and contraction for  graphs in pseudo-surfaces that is compatible with deletion and contraction for their associated matroid perspectives, we are able to give  complete deletion-contraction relations for the Las Vergnas polynomial.

 Given the three  extensions $L_G$,  $R_G$ and $K_G$ of the Tutte polynomial to embedded graphs, it is natural to ask how they are related. The Krushkal polynomial, $K_G$, contains both the embedded graph polynomials $L_G$ and $R_G$ as specialisations  (see \cite{ACEMS,But}), but does not yet provide a full understanding of the connection between the two polynomials.  We  similarly relate the Las Vergnas polynomial and Krushkal polynomials for (not necessarily cellularly embedded)  graphs in surfaces, and discuss  connections among these three topological Tutte polynomials.

 We also take the opportunity here to revisit some of Las Vergnas' work on Eulerian circuits.   In \cite{Las78}, Las Vergnas gave a number of formulae for enumerating Eulerian circuits of $4$-regular graphs  in  surfaces.  However, most of the formulas only apply for graphs in the sphere, torus, or real projective plane.  Now, with recently developed language and tools for ribbon graphs  we are able to extend these results to all surfaces.

 \section{Background on embedded graphs}
Our main aim here is to understand the Las Vergnas polynomial, which is defined as a polynomial of matroid perspectives, in the context of contemporary research in polynomials of embedded graphs (see also \cite{ACEMS} for work in this direction).  Doing so reveals its connections with other topological graph polynomials, exposes nuances of deletion and contraction, and facilitates future research.  In this section, we provide a brief review of some standard notation, assuming familiarity with basic graph theory and topological graph theory.  We pay particular attention to the language of ribbon graphs, since most of the recent research on topological graph polynomials appears in this context. Ribbon graphs will also be important in Section~\ref{s.lg}.
Further details of the material covered in this section may be found in \cite{EMMbook,GT87}.

 \subsection{Embedded graphs}\label{ss.eg}

% \subsubsection{Notation and terminology}
As usual, if $G$ is a graph, then  $V(G)$ is its vertex set, and $E(G)$ its edge set, with $v(G):=|V(G)|$ and  $e(G):=|E(G)|$. We denote the number of components of $G$ by $c(G)$.The rank of $G$ is $r(G):=v(G)-c(G)$, and the nullity of $G$ is $n(G):=e(G)-r(G)$.  These agree with the rank and nullity of the cycle matroid of the graph as discussed in Section~\ref{ss.matroidsnew}. 
If $A\subseteq E(G)$, then $v(A)$, $e(A)$, $c(A)$, $r(A)$, and $n(A)$ are the number of vertices, number of  edges, number of components, rank and nullity,  respectively, of the spanning subgraph $(V(G),A)$ of $G$. In cases where the graph $G$ may not be immediately clear from context, we will use a subscript writing, for example, $r_G(A)$.

\medskip

 Let $\Sigma$ be a connected surface or pseudo-surface (i.e., a surface with pinch points,  also known as a pinched surface), possibly with boundary. 
We use $k(\Sigma)$ to denote the number of connected components of the pseudo-surface $\Sigma$. 
  For a subset $X$ of $\Sigma$, we let  
  $N(X)$ denote a regular neighbourhood of $X$.  

If $\Sigma$ is a surface (without pinch points) its {\em Euler genus}, $\gamma(\Sigma)$, is its genus if it is non-orientable, and twice its genus if it is orientable. Recall that the {\em Euler characteristic}, $\chi(\Sigma)$, of $\Sigma$ can be obtained as   $\chi(\Sigma)=v_t-e_t+f_t$, where $v_t$, $e_t$, and $f_t$ are the numbers of vertices, edges, and faces, respectively, in any triangulation (or more generally, cellulation) of $\Sigma$.  {\em Euler's formula} gives that $ \gamma(\Sigma) = 2k(\Sigma)-b(\Sigma) -\chi(\Sigma)$, where  $b(\Sigma)$ is the number of the boundary components of $\Sigma$.

\medskip

A \emph{graph in a pseudo-surface},  $G\subset \Sigma$, consists of a graph $G$ and a drawing of $G$ on a pseudo-surface $\Sigma$   such that the edges only intersect at their ends and such that any pinch points are vertices of the graph.

The components of $\Sigma \backslash G$ are called the {\em regions} of $G$.  If each region of $G\subset \Sigma$  is homeomorphic to an open disc, it is said to be  {\em cellularly embedded} and the regions are called \emph{faces}.  Furthermore,  $G\subset \Sigma$ is a  \emph{cellularly embedded graph} if it is cellularly embedded and $\Sigma$ is a surface (so there are no pinch points). 
 If $G\subset \Sigma$ is a graph in the  pseudo-surface and $A\subseteq E(G)$  then we define $\rho(A)$ to be the number of regions of the spanning subgraph $(V(G) \cup  A) \subset \Sigma $ of $G\subset \Sigma$. That is, $ \rho(A) = k(\Sigma  \backslash (V(G) \cup  A))$. 
 If $G$ is not clear from context, we will specify it with a subscript, thus:  $\rho_G( A )$.

%\subsubsection{Surface deletion and contraction.} 
Deletion of an edge of a graph in a pseudo-surface is straight forward.  Given $G\subset \Sigma$ and $e\in E(G)$ then $G\backslash  e\subset \Sigma$ is the  graph in a pseudo-surface obtained by removing the edge $e$ from the drawing of $G\subset \Sigma$ (without removing the  points of $e$ from $\Sigma$, or its incident vertices).    Edge contraction is defined by forming a quotient space of the surface.  $G/e \subset \Sigma/e$ is the graph in a pseudo-surface obtained by identifying the edge $e$ to a point. This point becomes a vertex of $G/e$. 
Note that if $e$ is a loop, then contraction can create pinch points with the new vertex lying on it (see Figure~\ref{cont}).   For example,  if $G\subset \Sigma$ consists of a loop on a sphere, then  $G/  e\subset \Sigma$ consists of two spheres that meet at a pinch point, and that pinch point is a vertex.  Thus  the class of cellularly embedded graphs is not closed under either deletion or contraction.  
  At times it is convenient to view $G/e \subset \Sigma/e$ as the graph in a pseudo-surface that results from  removing a small open neighbourhood of $e$ from $\Sigma$, then identifying all boundary components that this creates to obtain a new vertex. 

An important observation for us here is that 
if $\widetilde{G}$ is the underlying abstract graph  of $G\subset \Sigma$, then the underlying abstract graph  of $G/e\subset \Sigma/e$ is $\widetilde{G}/  e$, similarly the underlying abstract graph  of $G\backslash e\subset \Sigma$ is $\widetilde{G}\backslash  e$.

\begin{figure}
\centering
\subfigure[$G\subset \Sigma$. ]{
\includegraphics[height=16mm]{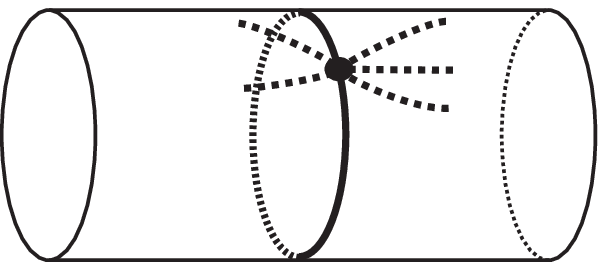}
\label{cont.a}
}
\hspace{1cm}
\subfigure[$G/e\subset \Sigma/e$.]{
\includegraphics[height=16mm]{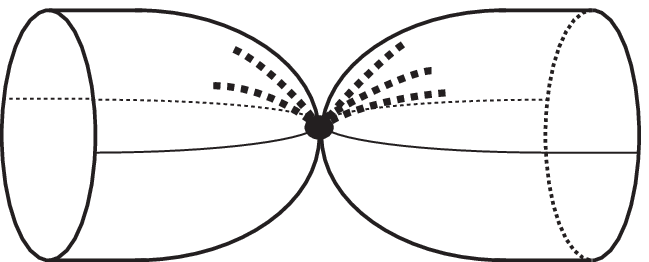}
\label{cont.c}
}
\caption{Contracting an orientable loop of a graph in a pseudo-surface.}
\label{cont}
\end{figure}

\subsection{Ribbon graphs}
%\subsubsection{Notation and terminology}
At times it will be convenient to describe cellularly embedded graphs as ribbon graphs. We refer the reader to \cite{EMMbook} for a more detailed discussion of ribbon graphs.   A {\em ribbon graph} $G =\left(  V(G),E(G)  \right)$ is a surface with boundary, represented as the union of two  sets of  discs: a set $V (G)$ of {\em vertices} and a set of {\em edges} $E (G)$ such that: (1) the vertices and edges intersect in disjoint line segments;
(2) each such line segment lies on the boundary of precisely one
vertex and precisely one edge;
(3) every edge contains exactly two such line segments.

Ribbon graphs arise as regular neighbourhoods of cellularly embedded graphs, with the vertex set of the ribbon graph arising from regular neighbourhoods of the vertices of the embedded graph, and the edge set of the  ribbon graph arising from regular neighbourhoods of the edges of the embedded graph.  On the other hand, if $G$ is a ribbon
graph, then topologically it is a surface with boundary.  Filling in
each hole by identifying its boundary component with the boundary of a disc results in a
ribbon graph embedded in a closed surface. (This in a {\em band decomposition}  with the vertex discs called {\em 0-bands}, the edge discs called {\em 1-bands} and the   face discs are called {\em 2-bands}.) A deformation retract of the ribbon graph in the surface yields a
graph cellularly embedded in the surface. Thus, ribbon graphs, band decompositions, and {\em cellularly} embedded graphs are equivalent.  (Note that this equivalence requires both that the graph is in a surface, rather than just a pseudo-surface, and that it is cellularly embedded.) 

If $G$ is a ribbon graph, then $G^*$ is the ribbon graph corresponding to the geometric dual when $G$ is viewed as a cellularly embedded graph.  It is often useful to think of $G$ and $G^*$ in the setting of band decompositions.   In this setting  they are represented by the same topological object, the only difference being that the sets designated as vertices (0-bands) and face discs (2-bands) are reversed.  Note that the dual of an isolated vertex is an isolated vertex.

 If $G$ is a ribbon graph, then $v(G)$, $e(G)$, $c(G)$, $r(G)$, and $n(G)$  are all as defined for the underlying abstract graph of $G$ (we use $c$ for the components of a graph and $k$ for the components of a surface, as these need not coincide if $G$ is not cellularly embedded). Furthermore,  $f(G)$ is the number of boundary components of the surface defining the ribbon graph, and  the  Euler genus, $\gamma(G) $, of  $G$ equals the Euler genus of the surface defining the ribbon graph. Since each boundary component of a ribbon graph corresponds to a face of a cellular embedding, Euler's formula gives that $v(G)-e(G)+f(G)=2c(G)-\gamma(G)$. A ribbon graph $G$ is   {\em plane}  if it is connected and $\gamma(G)=0$.  A ribbon graph $H $ is a {\em ribbon subgraph}  of $G$ if $H$ can be obtained by deleting vertices and edges of $G$. If $H$ is a ribbon subgraph of $G$ with $V(H)=V(G)$, then $H$ is a  {\em spanning ribbon subgraph} of $G$.  If $A \subseteq E(G)$, then $r(A)$, $c(A)$, $n(A)$,  $f(A)$, $\gamma(A)$ each refer to the spanning subgraph $(V(G),A)$ of $G$ (where $G$ is given by context).

Deletion for ribbon graphs just removes an edge: if $G$ is a ribbon graph, and $e\in E(G)$,  then $G - e$ is the spanning ribbon subgraph on edge set $E(G)\backslash \{e\}$.  (Note that we use ``$-$'' for ribbon graph edge deletion,    and ``$\backslash$'' for embedded graph edge deletion.)  An important aspect of deletion as defined via ribbon graphs is that the result is again a ribbon graph.  Remembering that ribbon graphs correspond to cellularly embedded graphs, ribbon graph edge deletion is appropriate for cellularly embedded graphs as  $G- e$ remains in the class of cellularly embedded graphs.  However, the surface associated with $G$ by filling in the holes with discs may not be the same surface that results from filling in the holes of with $G-e$. For example, if $e$ is a bridge or a non-orientable loop, then deleting the former will increase the number of components of the surface, and deleting the latter may change the graph's orientability.

\section{Matroid perspectives and Las Vergnas' polynomial} 
We will now review the original Las Vergnas polynomial of a cellularly embedded graph. This polynomial  arises as a special case of the Tutte polynomial of a matroid perspective. In this section we describe how it can be written in terms of parameters that are more frequently used in the study of topological graph polynomials. This allows for the polynomial to be  positioned properly in the field.  We will also discuss  deletion-contraction relations for the polynomial. While the Tutte polynomial of a matroid perspective has a complete recursive definition (complete in the sense that it reduces the computation of the polynomial to that of the trivial matroid perspective), the Las Vergnas polynomial does not. We will explain why this is the case at the end of this section, and will turn our attention fully to deletion-contraction reductions in Section~\ref{newLVsec}.

\subsection{Matroids and matroid perspectives} \label{ss.matroidsnew}
 Since the Las~Vergnas polynomial was originally defined in terms of the Tutte polynomial of a matroid perspective, we review the essential concepts of matroids and matroid perspectives, and recall the definition of the Tutte polynomial of a matroid perspective. We will work with matroids in terms of rank functions since this is most appropriate for the connections with graph polynomials.

A {\em matroid} $M=(E,r)$ consists of a set $E$ and a {\em rank function}, $r: \mathcal{P}(E) \rightarrow  \mathbb{Z}_{\geq 0}$, from the power set of $E$ to the non-negative integers such that for each $A\subseteq E$ and $e,f\in E$ we have
\begin{align}
&r(\emptyset)=0,  \label{max1} \\
&r(A\cup \{e\})  \in \{ r(A), r(A)+1\},  \label{max2}
\\
 &r(A)=r(A\cup \{e\})=r(A\cup \{f\}) \implies r(A\cup \{e,f\})=r(A).  \label{max3}
\end{align}

A set $A\subseteq E$ is {\em independent} if $r(A)=|A|$, and dependent otherwise.  It is a  {\em circuit} if it is a minimal dependent set, so in particular, if $A$ is a circuit, then  $r(A)=|A|-1$. A set $A$ is a \emph{flat} if for all $e\in E-A$ we have $r(A \cup e) = r(A)+1$.  An element  $e\in E$ is an {\em isthmus} (or {\em coloop})  if for each independent set $A$ we have that $A\cup \{e\}$ is also independent.  An element $e$ is a {\em loop} if $\{e\}$ is a circuit.

If $M=(E,r)$ is a matroid and $e\in E$, then $M\backslash e = (E\backslash \{e\}, r|_{E\backslash \{e\}} )$ is the matroid obtained by {\em deleting} $e$; and $M/e = (E\backslash \{e\}, r' )$, where $r'(A):= r(A\cup \{e\})-r(\{e\})$, is the matroid obtained by {\em contracting} $e$.
The {\em dual} of $M$ is the matroid given by  $M^*=(E, r^*)$, where $r^*(A):= |A|+r(E\backslash A) -r(E)$.

If $G$ is a graph,  its {\em cycle matroid} is   $C(G):=(E(G), r_{C(G)})  $, where  $r_{C(G)}(A):= v(A)-c(A)$; and its {\em bond matroid} is   $B(G):=(C(G))^*$.
When $G$ is a plane graph (i.e., a graph cellularly embedded a sphere) $ B(G^*)=  (C(G^*))^* = C((G^*)^*) =C(G)$. However, this identity {\em does not} hold in general when $G$ cellularly embedded in a higher genus surface.

\medskip

A {\em matroid perspective}   is a  triple $(M,M', \varphi)$ where  $M=(E, r)$ and $M'=(E',r')$ are matroids, and  $\varphi: E \rightarrow E'$ is a bijection such that  for all $A\subseteq B \subseteq E$, we have
\begin{equation}\label{e.mp}
 r(B)- r(A) \geq r'(\varphi(B))-r'(\varphi(A)).
\end{equation}

Following the usual convention in the area, at times we suppress the bijection $\varphi$ and use $M\rightarrow M'$ to denote a matroid perspective $(M,M', \varphi)$, especially since we will primarily be interested in matroid perspectives of the form $(M,M', \mathrm{id})$, i.e., where the ground sets are the same or may be naturally identified.  When $\varphi$ is the identity, we say $M \rightarrow M'$ is a matroid perspective on the set $E$.
In this case, as noted in \cite{Las80}, the condition given by \eqref{e.mp} can be equivalently formulated as the requirement that each circuit of $M$  is a union of circuits of $M'$, or as the requirement that every flat of $M'$ is a flat of $M$.  Thus, in particular, a loop of $M$ is a loop of $M'$ and an isthmus of $M'$ is an isthmus of $M$.

Deletion and contraction for a matroid perspective $(M,M', \varphi)$ are defined by, for $e\in E$, setting $(M,M', \varphi)\backslash e:=  (M\backslash e,M'\backslash e, \varphi|_{E\backslash e})$ and $(M,M', \varphi)/e:=  (M/e,M'/e, \varphi|_{E\backslash e})$. We will denote these matroid perspectives by  $M\backslash e\rightarrow M'\backslash e$ and  $M/e\rightarrow M'/e$, respectively.

\subsection{The Tutte polynomial of a matroid perspective}

Let $M=(E, r)$ and $M'=(E',r')$ be matroids.
As defined in \cite{Las78a,Las80}, the {\em Tutte polynomial} of the matroid perspective  $M\rightarrow M'= (M,M', \varphi)$ is defined by
  \begin{equation}
   T_{M\rightarrow M'}(x,y,z) = \sum_{X \subseteq E}
                 (x - 1)^{r'( E') - r'( \varphi(X) )}
                 (y-1)^{|X|-r(X)}
                  z^{(r(E)-r(X))-(r'(E')-r'(\varphi(X)))}.
  \end{equation}

As noted in \cite{Las80}, the classical Tutte polynomial $T_M(x,y) =   \sum_{X \subseteq E}{(x - 1)^{r( E) - r(X)} (y-1)^{|X|-r(X)}}$ of a matroid $M$ can be recovered from the more general polynomial:
\begin{align*}
T_M(x,y)& = T_{M\rightarrow M}(x,y,z), \\
T_M(x,y)& = T_{M\rightarrow M'}(x,y,x-1), \\
T_{M'}(x,y)& =  (y-1)^{r(M)-r(M')} T_{M\rightarrow M'}(x,y,1/(y-1)).
\end{align*}

Las Vergnas (Theorem 5.3 of \cite{Las99}) showed that $T_{M\rightarrow M'}$ satisfies deletion-contraction relations that provide a complete recursive definition of the polynomial.
\begin{theorem}\label{lvdelcont}
Let $M \rightarrow M'$ be a matroid perspective on a set $E$.  The following relations hold:
\begin{enumerate}
\item if $e \in E$ is neither an isthmus nor a loop of $M$, then
\[ T_{M\rightarrow M'}(x,y,z) = T_{M\backslash e\rightarrow M'\backslash e}(x,y,z)+T_{M/e\rightarrow M'/e}(x,y,z);\]
\item if $e \in E$ is an isthmus of $M'$, and hence also an isthmus of $M$, then
\[T_{M\rightarrow M'}(x,y,z) = x T_{M\backslash e\rightarrow M'\backslash e}(x,y,z);\]
\item if $e \in E$ is a loop of $M$, and hence also a loop of $M'$, then
\[ T_{M\rightarrow M'}(x,y,z)=yT_{M\backslash e\rightarrow M'\backslash e}(x,y,z); \]
\item if $e \in E$ is an isthmus of $M$, and is not an isthmus of $M'$, then
\[T_{M\rightarrow M'}(x,y,z) =zT_{M\backslash e\rightarrow M'\backslash e}(x,y,z) +T_{M/e\rightarrow M'/e}(x,y,z);\]
\item if $E=\emptyset$,   then $T_{M\rightarrow M'}(x,y,z) =1$.
\end{enumerate}
\end{theorem}

\subsection{Las Vergnas' topological Tutte polynomial}\label{ss.LVTPnew}

 The Las Vergnas  polynomial, $L_G$, was first defined in terms of the combinatorial geometry of an embedded graph, that is, $B(G^*)$ and $C(G)$, the bond and circuit geometries, or equivalently bond and cycle matroids, of $G^*$ and $G$ from Subsection \ref{ss.matroidsnew}.  In Proposition~\ref{our L}, we describe $L_G$ in graph theoretical terms. (This approach was also taken in \cite{ACEMS}.)  In this section, like Las Vergnas, we assume that $G$ is cellularly embedded, so that $G^*$ is also cellularly embedded in the same surface as $G$.  In this setting, we use the  rank function of $C(G)$ which is given by $r_G(A)=v(A)-c(A)$ for $A \subseteq E(G)$.

\begin{definition}
Let $G$ be a graph cellularly embedded in a surface $\Sigma$. Let  $ B(G^*) \rightarrow C(G)$ denote the matroid perspective  $(B(G^*), C(G), \mathrm{id})$, where $\mathrm{id}$ is the natural identification of the edges of $G$ and $G^*$ and so suppressed in the following. Then the {\em Las~Vergnas polynomial}, $L_G$ is defined by
\[   L_G(x,y,z)   :=   T_{B(G^*) \rightarrow C(G)}  (x,y,z).  \]
\end{definition}

By translating the notation and using  Euler's formula, we can rewrite Las Vergnas' topological Tutte polynomial in a form that more clearly reveals how it encodes topological information (see also \cite{ACEMS}).

\begin{proposition}\label{our L}
Let $G$ be a ribbon graph. Then
\begin{equation}\label{e.lv}
   \begin{split}
   L_G(x,y,z) &= \sum_{A \subseteq E( G)}
                 (x - 1)^{r_G(G) - r_G(A)}
                 (y-1)^{n_G(A)  -(\gamma(G)+\gamma_G(A)   -\gamma_{G^*}(A^c))/2   }    z^{ (\gamma(G)   -\gamma_G(A)   +\gamma_{G^*}(A^c))/2 },
   \end{split}
  \end{equation}
   where $A^c:=E(G)-A$.

\end{proposition}
\begin{proof}
By definition,  \begin{equation*}
   \begin{split}
   L_G(x,y,z) &= \sum_{A \subseteq E( G)}
                 (x - 1)^{r( C(G )) - r_{C(G)}( A )}
                 (y-1)^{|A|-r_{B(G^*)}(A)} z^{r(B(G^*))-r(C(G))-(r_{B(G^*)}(A)-r_{C(G)}(A))}. \\
   \end{split}
  \end{equation*}
Note that $r(C(G))=r(G)$ and  $r_{C(G)}(A)=r_G(A)$, and recall that $r_{M^*}(A) = |A|+r_M(M\backslash A) -r(M)$. Then, since $B(G)=(C(G))^*$, we have
$r(B(G^*))= r( C(G^*)^* )=  e(G^*)+r_{G^*}(\emptyset) -r(G^*)  = n(G^*) $, and
$r_{B(G^*)}  (A)=   r_{C(G^*)^*} (A) =  |A|+r_{G^*}(A^c) - r(G^*)  $.  (Recall that if $G$ is not plane, then $C(G^*)^*$ and $C(G)$ are not generally equal.)

 By Euler's formula and the facts that  $f_{G^*}(A^c)=f_G(A)$ and  $f(G^*)=v(G)$, we have
\begin{align*}
2r_{B(G^*)}  (A) & = 2|A|-2c_{G^*}(A^c)+2c(G^*) \\
 & = 2|A| -v_{G^*}(A^c) +|A^c|-f_{G^*}(A^c)-\gamma_{G^*}(A^c)+ v(G^*)-e(G^*)+f(G^*)+\gamma(G^*) \\
 &= |A| -v(G^*) +e(G^*)-f_{G^*}(A^c)-\gamma_{G^*}(A^c)+ v(G^*)-e(G^*)+f(G^*)+\gamma(G^*) \\
 &= |A|  -f_{G}(A)-\gamma_{G^*}(A^c)+v(G)+\gamma(G^*) \\
 &= v_G(A)-2c_G(A) +\gamma_G(A)   -\gamma_{G^*}(A^c)+v(G)+\gamma(G^*) \\
 &=  2r_G(A)+\gamma(G)+\gamma_G(A)   -\gamma_{G^*}(A^c).
\end{align*}
Then, using this computation for the exponent of $z$, we have:
\begin{align*}
r(B(G^*))-r(C(G))&-(r_{B(G^*)}(A)-r_{C(G)}(A))
\\
&=  n(G^*) -r(G) - r_G(A) +r_G(A)-(\gamma(G)+\gamma_G(A)   -\gamma_{G^*}(A^c))/2 \\
&=  e(G^*)-v(G^*)+c(G^*) - v(G) +c(G)  -(\gamma(G)+\gamma_G(A)   -\gamma_{G^*}(A^c))/2 \\
&=  e(G)-f(G)+2c(G) - v(G) -(\gamma(G)+\gamma_G(A)   -\gamma_{G^*}(A^c))/2 \\
&=  \gamma(G)  -(\gamma(G)+\gamma_G(A)   -\gamma_{G^*}(A^c))/2\\
&=  (\gamma(G)   -\gamma_G(A)   +\gamma_{G^*}(A^c))/2.
\end{align*}
 \end{proof}

If $G$ is plane, so that $\gamma(G)=\gamma(A)=0$ for all $A\subseteq E(G)$, then it is easily seen from Equation \ref{e.lv} that $ L_G(x,y,z) =T_G(x,y) $.
Furthermore, Las~Vergnas showed in \cite{Las80} that, for any cellularly embedded graph $G$, the Tutte polynomial of the underlying abstract graph of $G$ can be recovered from $L_G$ as
\begin{equation}\label{LtoT}
 (y-1)^{\gamma(G)} L_G(x,y,  1/(y-1) ) =T_G(x,y).
 \end{equation}

Collecting the topological contributions in the expression for $L_G$ given in Equation \ref{e.lv} gives the following particularly simple form of  $L_G$, which facilitates comparison with other topological graph polynomials.
\begin{equation}\label{tidyL}
  (z(y-1))^{\gamma(G)}   L_G\left( x,y, \frac{1}{z^2(y-1)}\right) =  \sum_{A \subseteq E( G)}
                 (x - 1)^{r_G(G) - r_G(A)}
                 (y-1)^{n_G(A) }    z^{ \gamma_G(A)   -\gamma_{G^*}(A^c) }.
\end{equation}

It is also informative to compare the following form of $L_G$, which is obtained by expanding the Euler genus terms using Euler's formula, to the dichromatic polynomial, $Z_G(x,y) := \sum_{A \subseteq E( G)} x^{c(A)}y^{|A|} = (x/y)^{c(G)}y^{v(G)}T_G((x+y)/x, y+1)$:
\[
L_G(x,y,z) = (1/(x-1)(y-1))^{c(G)} z^{n(G^*)} \sum_{A \subseteq E( G)} ((x-1)/z)^{c_G(A)}((y-1)z)^{c_{G^*}(A^c)}  (1/z)^{|A|}.
\]

We note that in \cite{CMNR} it is shown that $L_G$ is  determined by the delta-matroid of $G$, but we do not pursue this perspective here.

\subsection{Deletion-contraction}\label{ss.nodc}  
Although, by Theorem~\ref{lvdelcont}, the  Tutte  polynomial  of a matroid perspective $T_{M\rightarrow M'}$  has a deletion-contraction relation that applies to all types of edges, the Las~Vergnas polynomial  for cellularly embedded graphs (or ribbon graphs) does not (although it does have a deletion-contraction relation for some special types of edges). 
 Taking an example from \cite{ACEMS} a little further, if $G$ is the theta graph cellularly embedded on the torus, then $L_G(x,y,z)= 3z+2z^2+xz^2+1$.   Of the 17 cellularly embedded graphs on two edges, none have an $xz^2$ term, and so $L_G= 3z+2z^2+xz^2+1$ can not satisfy the deletion-contraction identities of Theorem~\ref{lvdelcont} for some  notion of edge deletion and contraction defined on the class of cellularly embedded graphs (such as ribbon graph deletion and contraction).   Thus although, for cellularly embedded graphs, $L_G$ is defined in terms of $T_{M\rightarrow M'}$, it does not inherit the recursive definition of $T_{M\rightarrow M'}$.

\section{The Las Vergnas polynomial for graphs in pseudo-surfaces} \label{newLVsec}  
The Las Vergnas polynomial of cellularly embedded graphs is not known to satisfy a complete recursive definition, even though its `parent', the Tutte polynomial of a matroid perspective, does. In this section we will describe how, by enlarging the domain of the Las Vergnas polynomial, the polynomial can be extended to a polynomial that does satisfy the deletion-contraction relations of Theorem~\ref{lvdelcont}. In keeping with the emphasis of this paper, we focus on  topological graph theoretic interpretations of the resulting polynomial.

\subsection{A generalised Las Vergnas polynomial.}
The construction of the matroid $B(G^*)$ used in the original definition of $L_G$  requires that $G$ be cellularly embedded so that the geometric dual $G^*$ is a well-defined cellularly embedded graph.  Essentially, what we want to extend the polynomial to  arbitrarily embedded graphs is a matroid that plays the role of the bond matroid of $G^*$ in the setting of non-cellularly embedded graphs. Mimicking the usual construction of $G^*$ by  placing a vertex in each connected component of the complement of $G$, and connecting vertices whose regions share an edge involves choices of how to embed the new edges in the surface that may result in inequivalent embeddings (although of the same abstract graph).  Nevertheless, we can construct an \emph{abstract} graph $G^{\dagger}$ from  $G \subset \Sigma$ whose bond matroid has the desired properties.
\begin{definition} Given a graph in a pseudo-surface, $G \subset \Sigma$, we define  $G^{\dagger}$ to be the abstract graph with vertex set corresponding to the regions of $\Sigma \backslash E$ and an edge between  (not necessarily distinct) vertices whenever the corresponding regions share an edge of $G$ on their boundaries (technically, the boundaries of regions meet boundaries of regular neighborhoods of edges, but the meaning is clear). Note that as in the case of geometric duals, this gives a natural identification between the edges of $G^{\dagger}$ and $G$.
\end{definition}

Given our interest in graph polynomials, our first aim is to understand the bond matroid $B(G^{\dagger})$ and its rank function in terms of parameters of the embedded graph.  
The following proposition is a generalisation of Lemma~4.1 of \cite{ACEMS}. 
\begin{proposition} \label{embedding component rank}
If $G \subset \Sigma$, then   $r_{B(G^{\dagger})}(A) = |A|  -\rho_G(  A )  +\rho_G( \emptyset)$.
\end{proposition}

\begin{proof}
We have
\begin{multline}\label{e.embedding rew}
r_{B(G^{\dagger})}(A) = r_{(C(G^{\dagger}))^*}(A) = |A|+r_{C(G^{\dagger})}(E\backslash A) -r_{C(G^{\dagger})}(E)
\\
= |A| + v_{G^{\dagger}}( E\backslash A ) -c_{G^{\dagger}}( E\backslash A ) -v_{G^{\dagger}}( E) +c_{G^{\dagger}}( E)
\\
= |A|  -c_{G^{\dagger}}( E\backslash A )  +c_{G^{\dagger}}( E)
=  |A|  -\rho_G(  A )  +\rho_G( \emptyset),
\end{multline}
For the last equality, if $G=(V,E)$ and  $R$ is the set of regions of $G$, then $c_{G^{\dagger}}( E\backslash A )$  is the number of components of $(R \cup (E\backslash A)) \backslash V$.  
On the other hand, $ \rho_G( A )$ is the number of components of  $\Sigma \backslash (V\cup A) =  (R \cup (E\backslash A)) \backslash V$ and it follows that  $c_{G^{\dagger}}( E\backslash A )= \rho_G( A )$. Taking $A=\emptyset$ in this argument gives 
 $c_{G^{\dagger}}( E) = \rho_G( \emptyset)$.
\end{proof}

\begin{theorem}\label{p.mat}
Let $G\subset \Sigma$ be a graph  in a pseudo-surface. Then
$ (B(G^{\dagger}), C(G), \mathrm{id}) $ is a matroid perspective, where $\mathrm{id}$ is the natural identification between the edges of $G$ and $G^{\dagger}$.
\end{theorem}
\begin{proof}
To prove that $ (B(G^{\dagger}), C(G), \mathrm{id}) $ is a matroid perspective, we must show that the rank functions satisfy the condition from  Equation~\eqref{e.mp}. By telescopic summations, it is suffices to do this one edge at a time, that is, to show that
\[ r_{B(G^{\dagger})}(A\cup\{e\})- r_{B(G^{\dagger})}(A) \geq r_{C(G)}(A\cup\{e\})-r_{C(G)}(A),\]
for each $A\subseteq E$ and $e\in E\backslash A$.
By Proposition \ref{embedding component rank}, this reduces to showing that 
\[1    - \rho_G(A\cup\{e\}) + \rho_G( A)  \geq  c_G(A)-c_G(A\cup\{e\}).  \]
Thus we need to show that
\[   c_G(A)-c_G(A\cup\{e\}) =1 \implies  \rho_G(A\cup\{e\}) -  \rho_G(A)  =0   .  \]
If $c_G(A)-c_G(A\cup\{e\}) =1$, then $e$ is a bridge of $(V(G),A)$. 
Also $ \rho_G( A \cup\{e\} )$ is the number of components of  $\Sigma \backslash (V\cup A\cup\{e\}) = (\Sigma \backslash (V\cup A))\backslash \{e\}  $,and  $ \rho_G( A )$ is the number of components of  $\Sigma \backslash (V\cup A) $. Since $e$ is a bridge of $(V(G),A)$, it bounds exactly one region of the drawing of $(V(G),A)$ on $\Sigma$. Thus deleting $e$ from $(\Sigma \backslash (V\cup A))$ will not create any additional connected components, giving $\rho_G(A\cup\{e\}) -  \rho_G(A)  =0 $, as needed.
\end{proof}

We can now extend the Las~Vergnas polynomial to all  graphs in pseudo-surfaces and obtain an expression for it in purely topological graph theory terms. 
\begin{definition} \label{newLVdef}
Let $G=(V,E)$ be a graph   in a pseudo-surface $\Sigma$. Then the {\em Las~Vergnas polynomial}, $L_{G\subset \Sigma}$, is defined by
\begin{multline} \label {LVeq}
L_{G\subset \Sigma}(x,y,z)   :=   T_ {(B(G^{\dagger}),  C(G), \mathrm{id})}  (x,y,z) \\=  \sum_{A \subseteq E} (x-1)^{c(A)-c(E)}(y-1)^{\rho(A)-\rho(\emptyset)}z^{|E|-|A|-\rho( E)+ \rho(A) + c(E)-c(A)}.
\end{multline}
\end{definition}

The following proposition tells us that this extended Las Vergnas polynomial does indeed specialize the original  Las Vergnas polynomial for cellularly embedded graphs.
\begin{proposition}
If  $G \subset \Sigma$ is a cellularly embedded graph, then $ (B(G^{\dagger}),  C(G), \mathrm{id})  =  (B(G^*),  C(G), \mathrm{id})$, and   $L_{G\subset \Sigma}(x,y,z) = L_G  $
\end{proposition}
\begin{proof}
The result follows immediately from the observation that if $G\subset \Sigma$ is a cellularly embedded graph, then $G^{\dagger}=G^*$ as abstract graphs.
\end{proof}

\subsection{Deletion and contraction}
Since one of the goals is a full deletion-contract reduction for $L_{G\subset \Sigma}$, we must first establish that deletion and contraction of an edge of $G\subset \Sigma$ is compatible with that for the bond matroid of $G^{\dagger}$.  (Since deletion and contraction do not change the underlying abstract graphs, we have that  $C(G)\backslash e =C(G\backslash e)$ and $C(G)/e =C(G/e)$.)
\begin{lemma}\label{l.mg}
Let $G\subset \Sigma$ be a graph in a pseudo-surface, and $e\in E(G)$. Then
\begin{enumerate}
\item  $B(G^{\dagger})\backslash e = B((G\backslash e)^{\dagger})$, and
\item  $B(G^{\dagger})/e = B((G/e)^{\dagger})$.
\end{enumerate}
\end{lemma}
\begin{proof} 
We use here the expression for the rank from Proposition \ref{embedding component rank} that $r_{B(G^{\dagger})}(A) = |A|  -\rho_G(  A )  +\rho_G( \emptyset)$.
 For the first item, note that $B(G^{\dagger})\backslash e$ and $B((G\backslash e)^{\dagger})$ are both on the same set, and  $r_{B(G^{\dagger})}(A)= r_{B((G \backslash e)^{\dagger})}(A)$ when $e\notin A$. This is since $\rho_G( \emptyset)-\rho_G( A)  = \rho_{G\backslash e}( \emptyset)- \rho_{G \backslash e} ( A)$ because if deleting $e$ from $G$ creates a new component this extra component is counted, with opposite signs,  in both $\rho_{G\backslash e}( \emptyset)$ and $\rho_{G \backslash e} ( A)$.

For the second item, again both matroids are on the same set.  For the rank functions, if $e \notin A$,
\begin{equation}\label{raco}
r_{B(G^{\dagger})/e} (  A )  =  r_{B(G^{\dagger})}(A\cup\{e\})- r_{B(G^{\dagger})}(\{e\}) =   |A|  - \rho_G(A\cup \{e\})+   \rho_G (\{ e\})   .
\end{equation}
On the other hand, by Proposition~\ref{embedding component rank}, 
\begin{equation}
r_{B((G/e)^{\dagger})} (  A )  =  |A|-  \rho_{G/e} (A)   + \rho_{G/e}(\emptyset).
\end{equation}
However,  since $(\Sigma /e) \backslash V(G/e)$ and $\Sigma  \backslash (V(G) \cup \{e\}) $ are homeomorphic (to see this, view contraction as the operation of removing a neighbourhood of $e$ then identifying all of the boundary components to a since vertex, as described in Section~\ref{ss.eg}), it follows that $\rho_{G/e}(\emptyset)=\rho_G (\{ e\})$ and  $\rho_{G/e} (A) = \rho_G(A\cup \{e\})$.
\end{proof}

If $G$ is a plane graph, we have that  $e$ is a loop in $G$ if and only if it is a loop in $B(G^*)$, and that $e$ is a bridge (cut-set of size one) in $G$ if and only if it is a isthmus in $B(G^*)$. This however does not hold for the matroid $B(G^{\dagger})$. To find the types of edges of a  graph in a pseudo-surface that correspond to loops and isthmuses in $B(G^{\dagger})$ we generalise loops and bridges by extracting one key feature of each:   
\begin{definition}
Let $G\subset \Sigma$ be a graph in a pseudo-surface, and $e\in E(G)$. Then we say that $e$ is a {\em quasi-loop} if  $\rho( e)>\rho(\emptyset)$, and we say that $e$ is a {\em quasi-bridge} if it is adjacent to exactly one region of $G\subset \Sigma$.
\end{definition}

A quasi-loop in $G$ is a loop, and a bridge in $G$ is a quasi-bridge. However, a loop is not necessarily a quasi-loop and quasi-bridge need not be a bridge (for example a longitudinal loop on a torus is not a quasi-loop, but is quasi-bridge). See Figure~\ref{f.venn}.

\begin{figure}
\centering
\labellist \small\hair 2pt
\pinlabel {loop}  at 90 135
\pinlabel {q.-loop}  at 70 80
\pinlabel {bridge}  at 215 80
\pinlabel {q.-bridge}  at 200 135
\endlabellist
\includegraphics[scale=0.5]{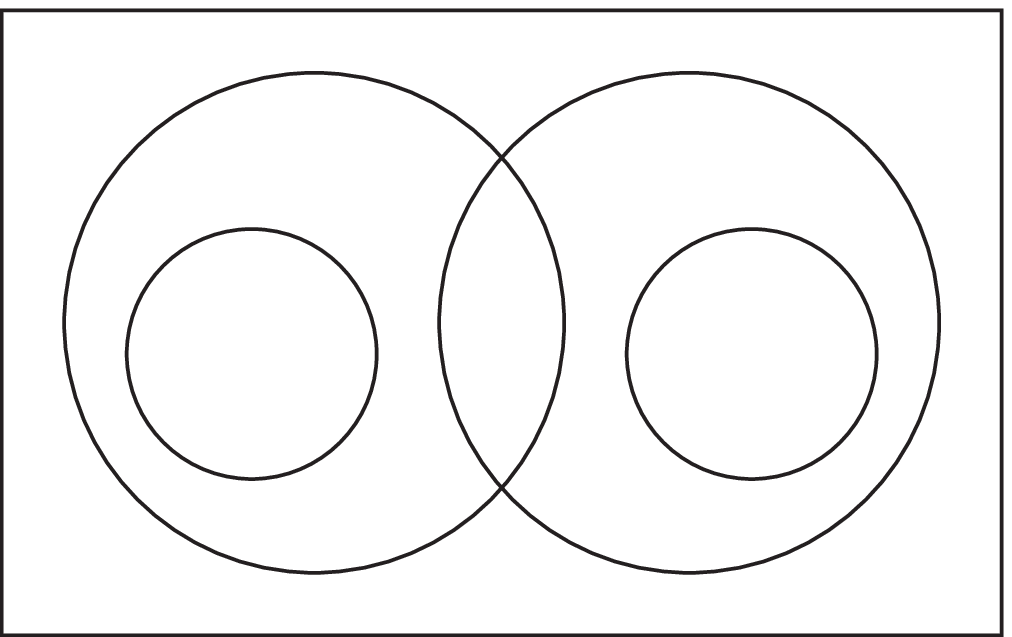}
\caption{A Venn diagram illustrating the various edge types.}
\label{f.venn}
\end{figure}

\begin{proposition}\label{l.mg2}
Let $G\subset \Sigma$ be a graph in a pseudo-surface, and $e\in E(G)$. Then the following hold.
\begin{enumerate}
\item $e$ is a quasi-loop in $G$  if and only if $e$ is a loop in $B(G^{\dagger})$.

\item  $e$ is a quasi-bridge in $G$  if and only if $e$ is an isthmus of $B(G^{\dagger})$.
\end{enumerate}
\end{proposition}
\begin{proof}  
For the first item, $e$ is a loop in $B(G^{\dagger})$  if and only if $r_{B(G^{\dagger})} (e)= 0$ if and only if $1-\rho(e)+\rho(\emptyset)=0$ if and only if $\rho( e)>\rho(\emptyset)$.

 For the second item we consider a slight generalisation of a polygonal decomposition of  a surface.  Recall that if $G$ is a cellularly embedded  in a surface $\Sigma$ we may create a polygonal decomposition of $\Sigma$ by arbitrarily orienting the edges of $G$ and giving them distinct labels (we identify the labels with the edge names).  The faces of $G$ are then polygons with directed labeled sides, and thus give a polygonal decomposition of $\Sigma$.  The surface $\Sigma$ may be recovered from this set of labeled polygons by identifying sides of polygons with the same labels consistently with the directions of the arrows.  This can be thought of as ``cutting'' the surface along the directed edges of $G$ to form the polygons with directed sides labeled by the edge names.  
 
Now let $G = (V,E) \subset \Sigma$ be a graph in a pseudo-surface.  We consider a slight generalisation of the above polygonal decomposition to obtain a decomposition of  $\Sigma \backslash V$. 
 As before direct and distinctly label the edges of $G$, delete small neighbourhoods of the vertices, and cut the surface along the edges.  The resulting regions are no longer necessarily polygons, but surfaces with labeled directed curves on their boundary components.  Again,  $\Sigma \backslash V$, with the edges of $G$ drawn on it, may be recovered by identifying curves with like labels so that the directions align.  We apply the construction to subsets of the edges as follows. Let $E:=E(G)$. If  $A\subseteq E$ then  $\cut(A)$ is the complex obtained by  arbitrarily  labelling and directing each edge of $G$, removing a small neighbourhood of each vertex, and   ``cutting $\Sigma$ open'' along the edges in $A$.  This results in a set of surfaces with boundaries, where the boundaries have directed arcs labeled by the edges of $A$ on them. These surfaces correspond to the regions of the spanning subgraphs $(V,A) \subset \Sigma$ of  $G \subset \Sigma$.  We have that   $ \Sigma \backslash V$ is obtained from $\cut(A)$ by, for each $a\in A$, identifying the two $a$-labelled boundary arcs such that their directions agree. We call the elements of $\cut(E)$ the {\em bricks} of $G\subset \Sigma$.

Observe that $|\cut(A)|= \rho(A)$, and that   $|\cut(A)|=\rho(\emptyset)$ if and only if $A$ is an independent set of $B(G^{\dagger})$. Any element of $\cut(A)$ may be formed from some subset of the bricks of $\cut(E)$ by identifying appropriate arcs labeled by edges in $E\setminus A$. Also observe that $e\in A$ if and only if $e$ labels two of the arcs on the boundary components of elements of $\cut(A)$. Moreover, $e$ is a quasi-bridge if and only if in $\cut(E)$ both of the boundary arcs labelled $e$ lie on the same brick.

We need to show that $e$ is an isthmus  of  $B(G^{\dagger})$ if and only if  $e$ is a quasi-bridge of $G\subset \Sigma$. That is, we need to show
\[  [ \rho( A) = \rho(\emptyset) \implies  \rho(A\cup \{e\})= \rho(\emptyset)   ]  \iff [e \text{ a quasi-bridge}].  \]
By restricting to the component of $\Sigma\backslash V$ that contains $e$, we can  assume without loss of generality that $\Sigma\backslash V$ is connected, and so it is enough to show that
\[  [ (\Sigma \backslash (V\cup A)  \text{ connected}) \implies  (\Sigma \backslash (V\cup A\cup \{e\})  \text{ connected})   ]  \iff [e \text{ is a quasi-bridge}].  \]
We will prove the equivalent statement,
\[  [ (\Sigma \backslash (V\cup A)  \text{ connected}) \text{ and } (\Sigma \backslash (V\cup A\cup \{e\})  \text{ not connected})   ]  \iff [e \text{ is not a quasi-bridge}].  \]

Suppose that  $\Sigma \backslash (V\cup A) $ is connected, but $\Sigma \backslash (V\cup  A\cup \{e\})$ is not. Then  $e$ lies on the boundary of exactly two elements of $\cut(A\cup \{e\})$, say $\mathcal{C}_1$ and $\mathcal{C}_2$.  Since $\mathcal{C}_1$ and $\mathcal{C}_2$ are formed from two disjoint sets of bricks of $\cut(E)$, it follows that $e$ must lie on two distinct bricks in $\cut(E)$, and is therefore not a quasi-bridge.

Conversely, suppose that $e$ is not a quasi-bridge. Then in $\cut(E)$ there are two distinct bricks $B_1$ and $B_2$ that have an arc labelled  $e$.
 Inductively construct a two component complex and a set of edges $A$  as follows.  Begin by setting $\mathcal{I}:=E$,  $\mathcal{C}_1:=B_1$, $\mathcal{C}_2:=B_2$ and $S:= \cut(E) \backslash \{B_1, B_2\}$.  As long as there is a label, say $b$, that only appears once on the boundary of $\mathcal{C}=\mathcal{C}_1 \sqcup \mathcal{C}_2$,  we update these sets with the following construction.  Notice that if there is such a $b$, then there is a brick $B \in S$ which also has a label of $b$ on its boundary.  Attach $B$ to $\mathcal{C}$ by identifying the $b$-labeled arcs.  Then let $\mathcal{C}_i$ be the result of the attachment if $B$ was attached to $\mathcal{C}_i$, and otherwise $\mathcal{C}_i=\mathcal{C}_i$.  Furthermore, let $\mathcal{I}:=\mathcal{I}\backslash \{b\}$ and $S:=S\backslash\{B\}$. Since there are only a finite number of edges, this process terminates.  When it does, $S=\emptyset$ since $\Sigma$ is connected. Let $\mathcal{C} = \mathcal{C}_1\cup \mathcal{C}_2$  denote the resulting complex, and $\mathcal{C}_e$ be the complex obtained from $\mathcal{C}$ by identifying the $e$-labelled arcs.

We then have that $ \mathcal{C}= \cut(\mathcal{I}) $ is not connected, but $\cut(\mathcal{I} \backslash e)$ is.   We set $A:=\mathcal{I} \backslash e$, and note that it is independent in $B(G^{\dagger})$, while $A\cup\{e\}$ is not. This completes the proof.
\end{proof}

Theorem~\ref{lvdelcont} together with Lemma~\ref{l.mg} and Proposition~\ref{l.mg2} now give the desired complete deletion-contraction relations for the Las~Vergnas polynomial:
\begin{theorem}\label{lvdelcont2}
Let $G\subset \Sigma$ be a graph in a pseudo-surface.  Then the following relations hold:
\begin{enumerate}
\item if $e \in E$ is neither an quasi-loop nor quasi-bridge of $G$, then
\[    L_{G\subset \Sigma}(x,y,z) =  L_{G\backslash e\subset \Sigma}(x,y,z)  +  L_{G/e\subset \Sigma/e}(x,y,z) ; \]

\item if $e \in E$ is a  bridge of $G$, then
\[   L_{G\subset \Sigma}(x,y,z) = x L_{G\backslash e\subset \Sigma}(x,y,z); \]

\item if $e \in E$ is a  quasi-loop of $G$, then
\[   L_{G\subset \Sigma}(x,y,z) = y L_{G\backslash e\subset \Sigma}(x,y,z); \]

\item if $e \in E$ is a quasi-bridge but not a bridge of $G$, then
\[    L_{G\subset \Sigma}(x,y,z) =  zL_{G\backslash e\subset \Sigma}(x,y,z)  +  L_{G/e\subset \Sigma/e}(x,y,z) ; \]

\item if $E(G)=\emptyset$,   then $L_{G\subset \Sigma}(x,y,z) =1$.
\end{enumerate}
\end{theorem}

Note that for plane graphs, loops are quasi-loops  and bridges are quasi-bridges, so, for plane graphs, the polynomial defined by the relations in Theorem~\ref{lvdelcont2} is indeed  the classical Tutte polynomial.

\section{Relations with other topological Tutte polynomials} \label{sec:polys}
In this section we  restrict to graphs in surfaces.
We consider two other notable topological Tutte polynomials, that is, polynomials of  graphs in surfaces that generalize the classical Tutte polynomial.  These are the 2002 ribbon graph polynomial of Bollob\'as and Riordan, $R_G$, from \cite{BR2} (which subsumes the 2001 version for orientable ribbon graphs from \cite{BR1}), and the 2011 Krushkal polynomial, $K_G$, from \cite{Kr} for graphs arbitrarily embedded in orientable surfaces (which was extended to non-orientable surfaces by Butler in \cite{But}).  We now determine the relations between these polynomials and the Las Vergnas polynomial, both the original version for cellularly embedded graphs (which was first done in \cite{ACEMS}), and the new version for arbitrarily embedded graphs.  We begin by recalling the definitions of $R_G$ and $K_G$.

   \begin{definition}\label{defBR}
Let $G$ be an cellularly embedded graph, or, equivalently, a ribbon graph.  Then the \emph{ribbon graph polynomial} or  {\em Bollob\'as-Riordan polynomial}, $R_G(x,y,z) \in \mathbb{Z}[x,y,z]$,  is defined by
\[   R_G(x,y,z) = \sum_{A \subseteq E( G)}   (x - 1)^{r_G( G ) - r_G( A )}   y^{n_G(A)} z^{c_G(A) - f_G(A) + n_G(A)}  . \]
  \end{definition}
Noting that exponent of $z$ is equal to the Euler genus $\gamma(A)$, the ribbon graph polynomial may be rewritten as

\begin{equation}\label{e.BRgenus1}
 R_G(x,y,z)=  \sum_{A \subseteq E( G)}  (x - 1)^{r_G(E) - r_G(A)}  y^{n_G(A) }    z^{ \gamma_G (A)} .
  \end{equation}

Although $R_G$ often appears with a fourth variable that records the orientability of each spanning  ribbon subgraph, here we omit it as it plays no role in our results.   Note that the classical  Tutte polynomial, $T_G$, is a specialisation of $R_G$ as
$ T_G(x,y)=R_G(x,y-1,1)$, and that  $T_G(x,y)=R_G(x,y-1, z)$ when $G$ is a plane graph (since when $G$ is plane  the Euler genus of all of its spanning ribbon subgraphs is zero).

Comparing the state sums  for $L_G$ and $R_G$ from Equations \eqref{tidyL} and \eqref{e.BRgenus1} illuminates the key differences and similarities between these two topological Tutte polynomials in the case of cellularly embedded graphs: $L(G)$ records information about the  spanning subgraphs of the dual, whereas $R(G)$ does not.  Furthermore, Equation \eqref{e.BRgenus1} together with Equation~\eqref{LtoT} gives that $R_G(x,y,1)= y^{\gamma(G)} L_G(x,y+1,1/y)$, when $G$ is cellularly embedded.  However, Askanazi et al. have given examples in \cite{ACEMS} suggesting  that it is unlikely that either of $R_G$ or $L_G$ may be recovered from the other.

\bigskip

We now turn our attention to the Krushkal polynomial which was defined in \cite{Kr} for graphs embedded (not necessarily cellularly)  in orientable surfaces, and in \cite{But} for graphs in non-orientable surfaces.
For this, recall from Section~\ref{ss.eg} that $N(X)$ denotes a regular neighbourhood of a subset $X$ of a surface $\Sigma$, and $k(\Sigma)$ is its number of connected components. 
The neighbourhood  $N(X)$ is itself a surface and so we can consider topological properties of this surface, such as its Euler genus.

\begin{definition}
Let $G\subset \Sigma$ be a graph in a surface. Then the {\em Krushkal polynomial}, $K_{G\subset \Sigma} ( x,y,a,b )\in \mathbb{Z}[a,b,x^{1/2},y^{1/2}]$, is defined by
\[  K_{G\subset \Sigma} ( x,y,a,b ) :=  \sum_{A\subseteq E(G)}  x^{c(G)-c(A)} y^{ k(\Sigma \backslash A)  - k(\Sigma)} a^{\frac{1}{2}\gamma(N(A))} b^{\frac{1}{2}\gamma(\Sigma \backslash A)}   .   \]
\end{definition}
We follow \cite{But} and use the form of the exponent of $y$ from the proof of Lemma~4.1 of \cite{ACEMS}  rather than the homological definition given in \cite{Kr}.

Krushkal showed for orientable surfaces \cite{Kr}, and Butler \cite{But} for non-orientable surfaces, that when $G\subset \Sigma$ is cellularly embedded, then the ribbon graph polynomial $R_G$ can be recovered from $K_{G\subset \Sigma} $ as
\[   R_G(x,y,z) = y^{\frac{1}{2}\gamma(G)}   K_{G\subset \Sigma}  (x-1,y,yz^2,y^{-1}).   \]
Furthermore, it was shown in \cite{ACEMS} for the orientable case, and \cite{But} for the non-orientable case, that the Las~Vergnas polynomial for cellularly embedded graphs can also be recovered from the Krushkal polynomial, here as
\begin{equation}\label{e.lvkrch}   L_G(x,y,z) = z^{\frac{1}{2}\gamma(G)}   K_{G\subset \Sigma}  (x-1,y-1,z^{-1},z).   \end{equation}
We can extend this relation to the full  Krushkal polynomial:
\begin{theorem}\label{t.lvkr}
Let  $G=(V,E)\subset \Sigma$ be a  graph in a surface.
Then
\begin{equation}\label{e.lvkr}  L_{(G, \Sigma)}(x,y,z) =  z^{\frac{1}{2}(\gamma(N(E)) -\gamma(\Sigma \backslash E))  }   K_{G\subset \Sigma}  (x-1,y-1,z^{-1},z).  \end{equation}
\end{theorem}
\begin{proof}
We proceed by comparing the exponents in the expression for $L_G$ from Definition \ref{newLVdef} with those on the right-hand side of Equation \eqref{e.lvkr}, which is
\begin{equation}\label{Keq}
\sum_{A \subseteq E}  {(x-1)^{c(A)-c(E)}(y-1)^{k(\Sigma \backslash A) -k(\Sigma)} z^{\frac{1}{2} (\gamma (\Sigma \backslash A) - \gamma(N(A)) + \gamma (N(G)) -\gamma (\Sigma \backslash E))}}.
\end{equation}

 The exponents of $x-1$ in Equations \eqref{LVeq} and \eqref{Keq} are the same. Since $\Sigma$ is a surface, deleting vertices does not change numbers of connected components,  and so the exponents of $y-1$ in Equations \eqref{LVeq} and \eqref{Keq} are the same. We now examine the exponents of $z$.

 Noting that $c(E)=k(N(V\cup E))$ and $c(A)=k(N(V\cup A))$, and since $\Sigma$ is a surface,  $\rho(E)= k(\Sigma\backslash E)$ and  $\rho(A) = k(\Sigma\backslash A)$ the exponent of $z$ in $L_{(G, \Sigma)}(x,y,z)$ is
\begin{equation}\label{eq.lvkr1}
 |E|-|A|-k(\Sigma\backslash E)+k(\Sigma\backslash A)+k(N(V\cup E))-k(N(V\cup A)).
\end{equation}

On the other hand, expanding the $z$ exponent  in Equation \eqref{Keq}  in terms of the Euler characteristic gives
\[\begin{array}{rclclcl}
&&k(N(V\cup E)) &-&\frac{1}{2}\chi(N(V\cup E))&-&\frac{1}{2}b(N(V\cup E)) \\
&-&k(\Sigma \backslash E) &+&\frac{1}{2}\chi(\Sigma \backslash E)&+&\frac{1}{2}b(\Sigma \backslash E)
\\ &  -&k(N(V\cup A)) &+&\frac{1}{2}\chi(N(V\cup A))&+&\frac{1}{2}b(N(V\cup A))
\\ &+&  k(\Sigma\backslash A) &-&\frac{1}{2}\chi(\Sigma\backslash A)&-&\frac{1}{2}b(\Sigma\backslash A). \\
\end{array}\]
The $b$ terms in this expression cancel since  $\Sigma\backslash A$ and $N(V\cup A)$ have identical boundary components for each $A\subseteq E(G)$.
To show that the above sum is equal to Equation~\eqref{eq.lvkr1} we show that $\chi(N(V\cup A)) -\chi(N(V\cup E))=|E|-|A|$ and that $\chi(\Sigma \backslash E)-\chi(\Sigma\backslash A)=|E|-|A|$.   It suffice to show this one edge at a time, i.e. to show that $ \chi(N(V\cup (X\cup \{e\}))) =   \chi(N(V\cup X))-1$, and
 $\chi(\Sigma \backslash (X\cup \{e\})) = \chi(\Sigma\backslash X)+1$,
 for any $X \subseteq E$.  This follows by recalling that   $\chi(\Sigma)=v_t-e_t+f_t$, where $v_t$, $e_t$, and $f_t$ are the numbers of vertices, edges, and faces, respectively, in any  cellulation of $\Sigma$, and then noting that extending a cellulation of $N(V\cup X)$ to $N(V\cup X \cup \{e\})$  changes the Euler characteristic by 1, as can easily be seen from Figure~\ref{f.2}.   Similarly,   $\chi(\Sigma \backslash E)-\chi(\Sigma\backslash A)=|E|-|A|$, and thus the exponents of $z$ in Equations \eqref{LVeq} and \eqref{Keq} agree, completing the proof.
\end{proof}

\begin{figure}
\centering
\hspace{1cm}
\subfigure[A triangulation of $N(V\cup X)$. ]{
\includegraphics[scale=.8]{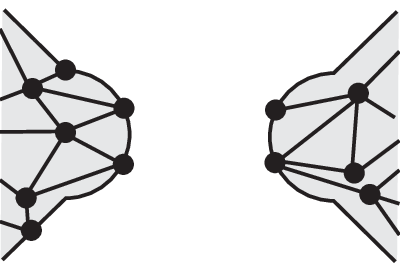}
\label{f.2b}
}
\hspace{2cm}
\subfigure[A triangulation of $N(V\cup X\cup\{e\})$. ]{
\includegraphics[scale=.8]{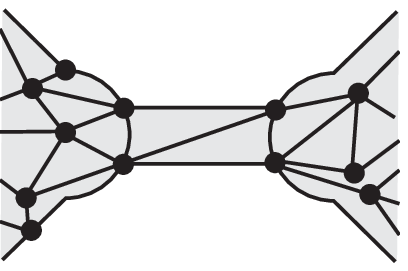}
\label{f.2c}
}
\caption{Triangulations of neighbourhoods of $(V,X)$ and $(V,X\cup \{e\})$ when $e\notin X$.}
\label{f.2}
\end{figure}

Observe that when $G\subseteq \Sigma$ is cellularly embedded then $\gamma(\Sigma \backslash E)=0$, and  Equations \eqref{e.lvkr} and \eqref{e.lvkrch}  agree.

It is likely that  the Bollob\'as-Riordan and Krushkal polynomials can be extended to graphs in pseudo-surfaces in such a way that the identity in Theorem~\ref{t.lvkr} still holds. We leave doing this as an open problem.

\section{New perspectives on Las~Vergnas' low genus work with Eulerian circuits}\label{s.lg}
Michel Las Vergnas also worked with cellularly embedded graphs via their Tait graphs.  We now use some tools recently developed to study twisted duality (see \cite{EMM2,EMMbook}) to build on Las Vergnas' foundations in this area.

 In this section we will work entirely with cellularly embedded graphs and ribbon graphs (which are equivalent). We recall that if $G=(V,E)$ is a ribbon graph and $A\subseteq E$ then $f(A)$ is the number of boundary components of the spanning ribbon subgraph $(V,A)$, and $\gamma(A)$ is its Euler genus. The parameters $f(A)$ and $\gamma(A)$ are most easily described in terms of ribbon graphs, but they can be computed in terms of cellularly embedded graphs: given $G\subset \Sigma$, describe $G$ as a ribbon graph,  construct its spanning ribbon subgraph $G'=(V,A)$, then translate back to the language of cellularly embedded graphs to get $G'   \subset \Sigma'$. Then $f(A)$ is  the number of faces of $G'$, and $\gamma(A)=\gamma(\Sigma')$. In particular, it is important to remember that $f(A)$ may  not be the number of regions of $G'\backslash A^c$, and similarly  $\gamma(A)$ need not equal $\gamma(\Sigma)$.

\subsection{Graph states and Tait graphs}

We first briefly recall some terminology.  Further details, including definitions of vertex and graph states, as well as medial and Tait graphs, relevant to this context,  may be found in \cite{EMM2,EMMbook}.

A \emph{vertex state} at a vertex $v$  of an abstract $4$-regular graph $F$ is a partition, into pairs, of the edges incident with $v$.
If $F$ is an cellularly embedded $4$-regular graph, a vertex state is simply the result of replacing a small neighbourhood of $v$  by a choice of one of the  configurations in Figure~\ref{c1.vstate.f1}.
 \begin{figure}[ht]
\centering
\begin{tabular}{ccccccc}
\labellist \small\hair 2pt
\pinlabel {$v$}  at 36 20
\endlabellist
\includegraphics[scale=.45]{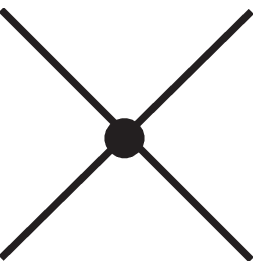}
 & \quad \raisebox{5mm}{$\longrightarrow$} \quad  & \includegraphics[scale=.45]{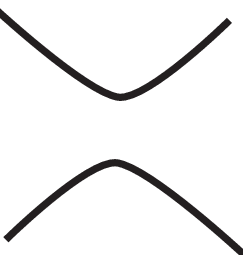} &\quad \raisebox{5mm}{,}\quad & \includegraphics[scale=.45]{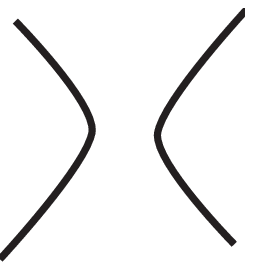} &\quad \raisebox{5mm}{or} \quad  &\includegraphics[scale=.45]{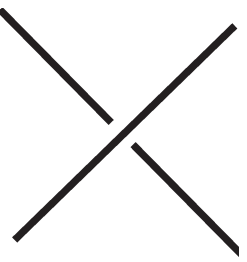} \\
\end{tabular}
\caption{The  vertex states of a vertex $v$ of a graph.}
\label{c1.vstate.f1}
\end{figure}

 If $G$ is a cellularly embedded graph and  $G_m$ its medial graph, checkerboard coloured so that faces containing a vertex of $G$ are coloured black, then we may use the checkerboard colouring to distinguish among the vertex states, naming them a {\em white split}, a {\em black split} or a {\em crossing}, as in Figure~\ref{c1.vstate.f2}.
\begin{figure}[ht]
\centering
\begin{tabular}{ccccccc}
\labellist \small\hair 2pt
\pinlabel {$v$}  at 36 20
\endlabellist
\includegraphics[scale=.45]{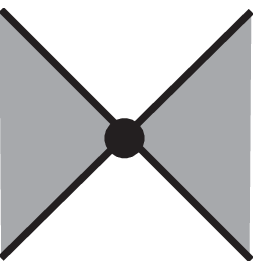}
 & \quad \raisebox{5mm}{$\longrightarrow$} \quad  & \includegraphics[scale=.45]{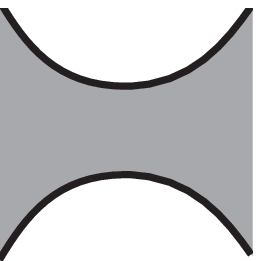} &\hspace{5mm}  & \includegraphics[scale=.45]{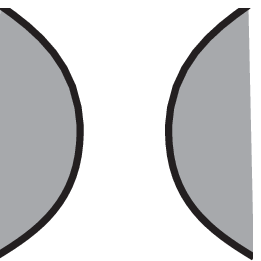} &\hspace{5mm}  &\includegraphics[scale=.45]{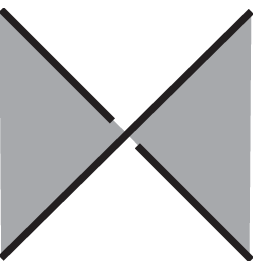} \\
in $G_m$ && white split && black split && crossing.
\end{tabular}
\caption{The three vertex states of a vertex $v$ of a  checkerboard coloured medial graph.}
\label{c1.vstate.f2}
\end{figure}

A {\em graph state} $s$ of  any 4-regular graph $F$ is a choice of vertex state at each of its vertices. Each graph state corresponds to a specific family of edge-disjoint cycles in $F$.  We call these cycles the {\em components of the state}, denoting the number of them by $c(s)$.

A cellularly embedded graph $G$ is a {\em Tait graph} of a cellularly embedded $4$-regular graph $F$ if $F$ is the medial graph of $G$. A cellularly embedded checkerboard colourable 4-regular graph is always a medial graph and will have exactly two (possibly isomorphic) Tait graphs, one corresponding to each colour in the checkerboard colouring as in the following definition. We will generally view Tait graphs as ribbon graphs.
\begin{definition}\label{c1.s5.ss2.d1}
Let  $F$ be a checkerboard coloured $4$-regular cellularly embedded graph. Then
\begin{enumerate}
\item the \emph{blackface graph}, $F_{bl}$, of $F$ is the  embedded graph constructed  by  placing one vertex in each black face and adding an edge between two of these vertices whenever the corresponding regions meet at a vertex of $F$;
\item the \emph{whiteface graph}, $F_{wh}$, is constructed analogously by placing vertices in the white faces.
\end{enumerate}
\end{definition}
Note that $F_{bl}$, $F_{wh}$ are the two Tait graphs of $F$, and that choosing the other checkerboard colouring just switches the names of  $F_{bl}$ and $F_{wh}$. An example is given in Figure~\ref{c1.f6}.

\begin{figure}
\centering
\subfigure[A checkerboard coloured 4-regular embedded graph $F \subset S^2$. ]{
\quad \quad\includegraphics[scale=.8]{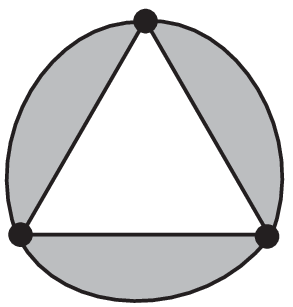}\quad\quad
\label{c1.f6b}
}
\hspace{10mm}
\subfigure[The blackface graph $F_{bl} \subset S^2$.]{
 \quad\includegraphics[scale=.9]{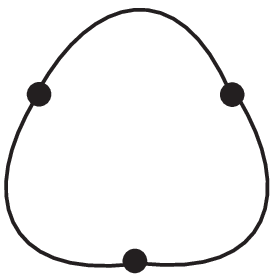}\quad
\label{c1.f6d}
}
\hspace{10mm}
\subfigure[The whiteface graph $F_{wh} \subset S^2$.]{
 \quad\includegraphics[scale=.8]{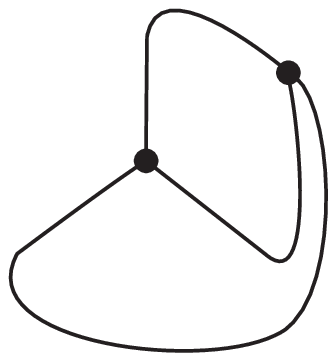}\quad
\label{c1.f6f}
}
\caption{The blackface and whiteface graph. Note that $F$ is the medial graph of both its blackface graph and its whiteface graph.}
\label{c1.f6}
\end{figure}

\subsection{Circuits in medial graphs}\label{ss.lv1}

We begin with the main theorem of \cite{Las78} which is a formula for the number of components in a graph state without crossings of a checkerboard coloured 4-regular graph (or equivalently, a checkerboard coloured medial graph) cellularly embedded in the sphere, torus, or real projective plane. We note that in the language of \cite{Las78}, a graph state with $k$-components is called an Eulerian $k$-partition. Also, the labelling of vertex states as black or white in \cite{Las78} is the reverse  from that used in this paper.

\begin{theorem}[Las Vergnas \cite{Las78}]\label{LV components}  Let $F$ be a checkerboard coloured $4$-regular graph cellularly embedded in the sphere, torus, or real projective plane; and let $s$ be a graph state without crossings.  Then the number of components of $s$  is equal to
\begin{equation}\label{LV noncrossing}
\min \{ (|B| +r(F_{wh})-2r_{F_{wh}}(B)+1), \quad (v(F)-|B|+ r(F_{bl}) -2r_{F_{bl}}(W)+ 1) \},
\end{equation}
where $B$ is the set of edges of $F_{wh}$ corresponding to vertices of $F$ with a black split in the graph state, and where $W$ is the set of edges of $F_{bl}$ corresponding to vertices of $F$ with a white split in the graph state when we view $F$ as the medial graph of both $F_{wh}$ and $F_{bl}$.
\end{theorem}

The strength of this formula is that it computes a topological property from  readily attainable combinatorial quantities.

We now give a related formula for the number of components of a graph state, with a much shorter proof than the original, that does hold for every surface.  We then use it to explain why the formula of Theorem \ref{LV components} fails on surfaces other than the sphere, torus, or real projective plane.

\begin{proposition}\label{our components}
Let $F$ be a $4$-regular connected checkerboard coloured cellularly embedded graph, and let $s$ be a graph state without crossings.  Then the number of components in $s$ is
\begin{equation}\label{our noncrossing}
f_{F_{bl}}(W)=
2c_{F_{bl}}(W)-\gamma_{F_{bl}}(W) +|W|-v(F_{bl}),
\end{equation}
where $F_{bl}$ is viewed as a ribbon graph, and $W$ is the set of edges of $F_{bl}$ corresponding to vertices of $F$ with a white split in the graph state $s$.
\end{proposition}

\begin{proof}
This is nearly a tautology.  We see in Figure \ref{taut2}  an edge of $F_{bl}$  (realised as a ribbon graph) together with the corresponding vertex of $F$, which shows that black splits essentially `snip through' the corresponding edges, effectively deleting them.
Thus, the components of the graph state $s$ of $F_{bl}$ just follow the face boundaries when the edges corresponding to black splits are deleted. The number of circuits in a state with no crossings is then just $f_{F_{bl}}(W)$.  The right-hand side of Equation (\ref{our noncrossing}) follows from Euler's formula.
\end{proof}

\begin{figure}[ht]
\centering
\subfigure[The edge $e$ in $G$ and corresponding vertex in $G_m$. ]{
\labellist \small\hair 2pt
\pinlabel {$e$}  at     125 67
\pinlabel {$v_e$}  at    69 67
\endlabellist
\quad\includegraphics[height=20mm]{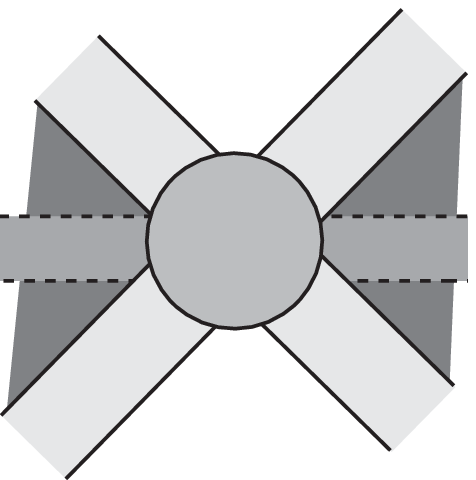}\quad
\label{taut2a}
}
\hspace{10mm}
\subfigure[Black split: `snips' $e$.]{
\quad \includegraphics[height=20mm]{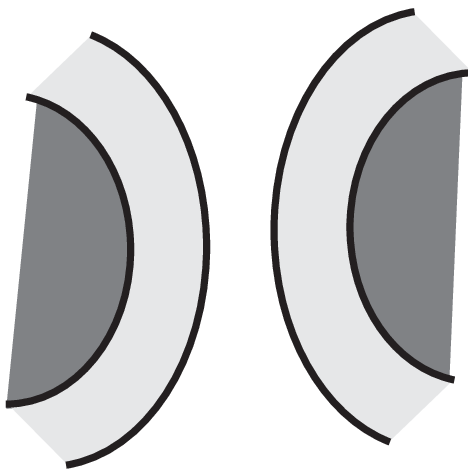}\quad
\label{taut2b}
}
\hspace{10mm}
\subfigure[White split: follows the boundaries of $e$.]{
\quad\includegraphics[height=20mm]{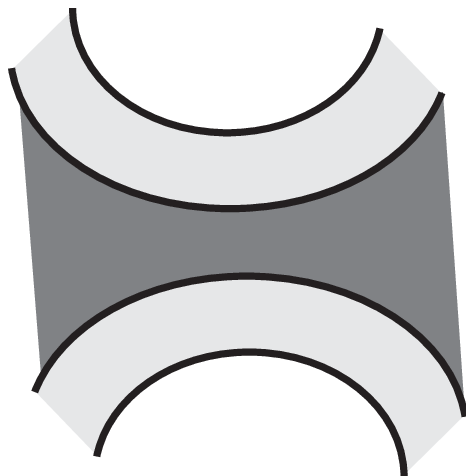}\quad
\label{taut2c}
}
\hspace{10mm}
\subfigure[Crossing: follows the boundaries of a half twist of $e$.]{
\quad\includegraphics[height=20mm]{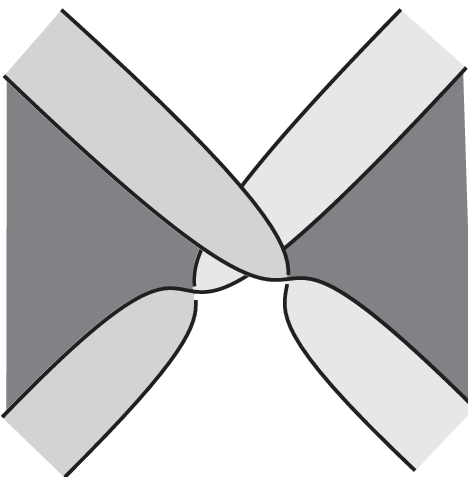}\quad
\label{taut2d}
}
\caption{For the justification of Propositions~\ref{our components} and \ref{our components2}.}
\label{taut2}
\end{figure}

Although, since practically tautological, Proposition \ref{our components} may be less useful than Theorem \ref{LV components}, it does lead us to rewrite Theorem \ref{LV components} in a form that reveals why the theorem does not generalise to other surfaces.

\begin{theorem}\label{LV components2} Let $F$ be a connected checkerboard coloured 4-regular graph cellularly embedded in the sphere, torus, or real projective plane.  Then the number of components of a graph state without crossings is equal to
\begin{equation}\label{LV noncrossing2}
\min \{ f_{F_{bl}}(W) + \gamma_{F_{wh}}(B), \quad f_{F_{bl}}(W) + \gamma_{F_{bl}}(W) \},
\end{equation}
where $F_{bl}$ and $F_{wh}$ are viewed as ribbon graphs,  $B$ is the set of edges of either $F_{bl}$ or $F_{wh}$ corresponding to vertices of $F$ with a black split in the graph state, and where $W$ is the set of edges of either $F_{bl}$ or $F_{wh}$ corresponding to vertices of $F$ with a white split in the graph state, when we view $F$ as the medial graph of both $F_{wh}$ and $F_{bl}$.
\end{theorem}

\begin{proof}
Viewing $F_{bl}$ and $F_{wh}$ as  ribbon graphs, Euler's formula states that $v(G)-e(G)+f(G)  =2c(G)-\gamma(G)$.
  With this,
\begin{equation*}
\begin{split}
|B| +r(F_{wh})-2r_{F_{wh}}(B)+1
&=|B|+v(F_{wh})-2v(F_{wh})+2c_{F_{wh}}(B)\\
&= f_{F_{wh}}(B) + \gamma_{F_{wh}}(B) \\
&= f_{F_{bl}}(W) + \gamma_{F_{wh}}(B),
\end{split}
\end{equation*}
where the  last equality follows by noting that since $B$ and $W$ are complementary sets in dual graphs, $f_{F_{bl}}(W)=f_{F_{wh}}(B)$.
A similar calculation shows that $v(F)-|B|+ r(F_{bl}) -2r_{F_{bl}}(W)+ 1 =f_{F_{bl}}(W) + \gamma_{F_{bl}}(W)$, and the result then follows by Theorem~\ref{LV components}.
\end{proof}

In the proof of Corollary \ref{low genus} we can now see the importance of low genus in Theorem \ref{LV components}.
\begin{corollary}\label{low genus}
If  $F$ is a connected checkerboard coloured 4-regular graph cellularly embedded in the sphere, torus, or projective plane, then
\[
\min \{ f_{F_{bl}}(W) + \gamma_{F_{wh}}(B), \quad f_{F_{bl}}(W) + \gamma_{F_{bl}}(W) \}=f_{F_{bl}}(W),
\]
where $F_{bl}$ and $F_{wh}$ are viewed as ribbon graphs,  and $B$ and $W$ are as in the statement of Theorem~\ref{LV components2}.
\end{corollary}

\begin{proof}
For the plane, torus, or projective plane, we note that $\gamma_{F_{wh}}(B)$ and $\gamma_{F_{bl}}(W)$ are in $\{0,1,2\}$.  For the plane, both are 0, so the result follows immediately.  On the torus and the projective plane, since $(F_{wh}- W)$ and $(F_{bl}- B)$ are edge disjoint (if we identify the edges of $F_{wh}$ and $F_{bl}$), both cannot contain fundamental cycles.  Thus, one or the other of $\gamma_{F_{wh}}(B)$ and $\gamma_{F_{bl}}(W)$ must be 0, from which the result follows. This  is not the case  on surfaces of higher genus.
\end{proof}

The tools of twisted duality from \cite{EMM2,EMMbook} allow us to extend the enumeration formula in Proposition \ref{our components} to all graph states, not just those without crossings.  We will not review those tools in detail here, but only note that an edge in a ribbon graph may be given a ``half-twist'', i.e. detach one end of an ribbon from an incident vertex, give the ribbon a half twist, and then reattach it.  If $G$ is a ribbon graph, and $A \subseteq E(G)$, then $G^{\tau(A)}$ is the ribbon graph resulting from giving a half-twist to all the edges in $A$.

\begin{proposition}\label{our components2}
Let $F$ be a connected checkerboard coloured $4$-regular cellularly embedded graph.  Then the number of circuits in any graph state is
\begin{equation*}\label{our crossing}
f((F_{bl})^{\tau(C)}- B),
\end{equation*}
where $F_{bl}$ is viewed as a ribbon graph,  $B$ is the set of edges of $F_{bl}$ corresponding to vertices of $F$ with a black split in the graph state, and $C$ is the set corresponding to crossings.
\end{proposition}

The proof, based on Figure~\ref{taut2} is  nearly a tautology, so we omit it.

\medskip
Las~Vergnas provided, in Theorem~\ref{LVtrees} below, an application of Theorem~\ref{LV components}   which relates Eulerian circuits and spanning trees.  By using the language of ribbon graphs and the quasi-bridges introduced in the previous section, we can now extend this result and give new perspectives on circuits in medial graphs.

\begin{theorem}[Las Vergnas \cite{Las78}]\label{LVtrees}
Let $F$ be a  checkerboard coloured $4$-regular graph embedded in the sphere, torus or real projective plane. Let $s$ be a graph state without crossings of $F$, let  $B$ be  the set of edges of $F_{wh}$ corresponding to vertices of $F$ with a black split in the graph state $s$, and $W$ be the set of edges of $F_{bl}$ corresponding to vertices of $F$ with a white split in the graph state $s$. Then $s$ defines an Euler circuit of $F$ if and only if $F_{wh}- W$ is a spanning tree of $F_{wh}$, or $F_{bl}- B$ is a spanning tree of $F_{bl}$.
\end{theorem}

The language of ribbon graphs allows us to extend Theorem~\ref{LVtrees} to all cellularly embedded graphs.  If we let $G$ denote the whiteface graph $F_{wh}$ and view it as a ribbon graph, then an Eulerian circuit without crossings in $F$ corresponds to a {\em quasi-tree} of $G$, which is a  ribbon subgraph of $G$ that has exactly one face (so all the edges of a quasi-tree are quasi-bridges). In addition, the ribbon graph $F_{wh}- W$ corresponds to a ribbon subgraph $G- A$ of $G$, and $F_{bl}- B$ corresponds to a ribbon subgraph $G^*- A^c$ of $G^*$, where  we identify the edges of $G$ and $G^*$, and $A^c=E(G)- A=E(G^*)- A$. Thus, Theorem~\ref{LVtrees} is equivalent to the statement that if $G$ is a ribbon graph  homeomorphic to a punctured sphere, torus  or real projective plane, then $G- A$ is a quasi-tree if and only if $G- A$ or $G^*- A^c$ is a spanning tree of $G$. It is clear that this statement, and hence Las~Vergnas' Theorem~\ref{LVtrees}, is completed by Theorem~\ref{LVpt} below.

\begin{theorem}\label{LVpt}
Let $G$ be a ribbon graph and $A\subseteq E(G)$. Then $G- A$ is a quasi-tree if and only if $G^*- A^c$ is a quasi-tree. Moreover, if $G- A$ is a quasi-tree, then
 \[  \gamma_G(A^c)+\gamma_{G^*}(A) = \gamma(G).  \]
 \end{theorem}
  \begin{proof} Since the ribbon subgraphs $G - A$ and $G^*- A^c$ of $G$ have the same boundary components, $G- A$ has exactly one face  if and only if $G^*- A^c$ has exactly one face. This proves the first part of the theorem.

For the second statement, suppose that $G- A$ is a quasi-tree. It then follows that $G$, $G- A$ and $G^*|_A$ are all connected.  By Euler's formula we then have \begin{equation*}
\begin{split}
 \gamma_G(A^c)+\gamma_{G^*}(A) =&\gamma( G- A) +\gamma(G^*- A^c) \\
=& e( G- A)-v( G- A)-f( G- A)+2c( G- A)  \\&+ e(G^*- A^c)-v(G^*- A^c)-f(G^*- A^c)+2c(G^*- A^c)\\
=& e(G)-v(G)-f(G) +2 \\
=& e(G)-v(G)-f(G)+2c(G)=\gamma(G) ,
\end{split}\end{equation*} where the second equality follows since $e(G)=e( G- A)+e(G^*- A^c)$, $v( G- A)=v(G)$, $v(G^*- A^c)=v(G^*)=f(G)$, and $f( G- A)=f(G^*- A^c)=c( G- A)=c(G^*- A^c)=c(G)=1$ (as $G- A$ is a quasi-tree).
\end{proof}

\subsection{A curious relation}

In \cite{Las78}, Las Vergnas also gave interpretations for evaluations of $L_G$ for graphs cellularly embedded in the plane, torus, or real projective plane in terms of the medial graph of $G$.  We conclude by showing that this now yields a very different kind of relationship between the Las Vergnas polynomial and the Bollob\'as-Riordan polynomial than that given previously in Section \ref{sec:polys}.  This identity uses circuits in medial graphs to give a relation between one variable specialisations of $L_G$ and $R_G$ on low genus graphs.
To do so, we first note the following evaluation of $R_G$.
\begin{proposition}\label{noncrossgenfunct}
Let $G$ be a connected cellularly embedded graph
 and let $f_k(G_m)$ be the number of $k$-component graph  states of its medial graph $G_m$ without crossings.  Then
\[
t R_G( t+1, t, 1/t)=\sum_{k \geq 1}f_k(G_m) t^k.
\]
\end{proposition}

\begin{proof}
This result is immediate from the relation between the topological transition polynomial and $R_G$ in \cite{EMS11},  but can also be seen as follows.  If $G$ is connected, then $R_G( t+1, t, 1/t)=t^{-1} \sum_{A \subseteq E(G)} t^{f(A)}$.  Note that there is a one-to-one correspondence between the boundary components of the spanning ribbon subgraphs of $G$ and the components of states of $G_m$ with no crossings.  This correspondence is given by a white split at the vertex corresponding to an edge $e$ if $e \in A$, and a black split otherwise.  Thus,  $R_G( t+1, t, 1/t)= t^{-1}\sum t^{c(s)}$, where the sum is over all non-crossings states $s$ of $G_m$,  and collecting like terms gives the result.
\end{proof}

\begin{theorem}\label{t.lrrelate}
 If $G$ is a graph embedded on the plane or real projective plane, then
\[ L_G(t+1, t+1, 1)=R_G(t+1, t, 1/t);\]
and if $G$ is embedded in the torus, then
\[L_{2,G}(t+1, t+1)+tL_{1,G}(t+1, t+1)+L_{0,G}(t+1, t+1)=R_G( t+1, t, 1/t) , \]
where, if we view $L_G(x,y,z)$ as a polynomial in $\left(\mathbb{Z}[x,y]\right)[z]$, then $L_{i,G}(x,y)$ is the coefficient of $z^i$ in $L_G(x,y,z)$.
\end{theorem}
\begin{proof} Let $F$ be the checkerboard coloured medial graph of $G$ so that $F_{bl}=G$.
Las Vergnas proved in Proposition~4.1 of \cite{Las78} that $tL_{F_{bl}}(t+1, t+1, 1)=\sum_{k \geq 1}f_k(F) t^k$ when $F$ is on the sphere or real projective plane; and that
$L_{2, F_{bl}}(t+1, t+1)+tL_{1,F_{bl}}(t+1, t+1)+L_{0,F_{bl}}(t+1, t+1)=\sum_{k \geq 1}f_k(F) t^{k-1}$, when $F$ is on the torus.  The results then follow by Proposition~\ref{noncrossgenfunct}.
\end{proof}

\section*{Acknowledgements}
The work of the first  author was supported by the National Science Foundation (NSF) under grants  DMS-1001408 and EFRI-1332411. We  thank the anonymous referees for several valuable suggestions. This work was completed while the authors were visiting the Erwin Schr\"odinger International Institute for Mathematical Physics (ESI) in Vienna. We thank the ESI for their support and for providing a productive working environment.
\bibliographystyle{amsplain}

\end{document}